\documentclass[12pt]{amsart}
\usepackage{amssymb,latexsym}
\usepackage{amsfonts}
\usepackage{amsmath}
\usepackage[colorlinks,linkcolor=blue,anchorcolor=blue,citecolor=blue]{hyperref}
\usepackage{algorithm}
\usepackage{enumerate}
\usepackage{algpseudocode}
\usepackage{verbatim}
\usepackage{graphicx}

\newcommand\C{{\mathbb C}}

\newcommand\Q{{\mathbb Q}}

\newcommand\Z{{\mathbb Z}}
\newcommand\N{{\mathbb N}}

\newcommand\cN{{\mathcal{N}}}

\newcommand\cD{\mathcal{D}}
\newcommand\al{\alpha}
\newcommand\be{\beta}
\newcommand\eps{\varepsilon}

\newcommand\bk{\textbf{k}}
\newcommand\lcm{\mathrm{lcm}}

\newtheorem{theorem}{Theorem}[section]
\newtheorem{lemma}[theorem]{Lemma}

\newtheorem{corollary}[theorem]{Corollary}

\newtheorem{conjecture}[theorem]{Conjecture}

\theoremstyle{remark}

\numberwithin{equation}{section}

\begin{document}

\title[Polynomials with multiplicatively dependent roots]{On the number of integer polynomials with multiplicatively dependent roots}

\author{Art\= uras Dubickas}
\address{Institute of Mathematics, Faculty of Mathematics and Informatics, Vilnius University, 
Naugarduko 24, Vilnius LT-03225,  Lithuania}
\email{arturas.dubickas@mif.vu.lt}

\author{Min Sha}
\address{Department of Computing, Macquarie University,
Sydney, NSW 2109, Australia}
\email{shamin2010@gmail.com}



\begin{abstract}
In this paper, we give some counting results on  integer polynomials of fixed degree and bounded height whose distinct non-zero roots are multiplicatively dependent. 
These include sharp lower bounds, upper bounds and asymptotic formulas for various cases, although in general there is a logarithmic gap between lower and upper bounds. 
\end{abstract}

\maketitle

\section{Introduction}  \label{sec:int}

\subsection{Motivation}

Let $n\ge 2$ be a positive integer. For non-zero complex numbers $z_1,\dots,z_n \in \C^*$, we say that they are \textit{multiplicatively dependent} (resp. \textit{linearly dependent})
if there exists a non-zero integer vector $(k_1,\dots,k_n) \in \Z^n$ for which
$$
z_1^{k_1}\cdots z_n^{k_n}=1
$$ 
(resp. 
$k_1z_1 + \cdots + k_nz_n=0$).

Throughout, the \textit{height} of a complex polynomial in $\C[X]$ is defined to be the largest modulus of its coefficients. 
For a polynomial $f\in \C[X]$ of degree at least two, we say that $f$ is \textit{degenerate} if it has a pair of distinct non-zero roots whose quotient is a root of unity. 
In \cite{dusa1}, the same authors have established sharp bounds for the number of degenerate integer polynomials of fixed degree and bounded height. 

Here, we say that $f$ is a \textit{generalized degenerate polynomial} if its distinct non-zero roots are multiplicatively dependent. 
Clearly, if $f$ is degenerate, then it is also a generalized degenerate polynomial. 
Our aim is to estimate the number of generalized degenerate integer polynomials of fixed degree and bounded height. 
Our results show that these polynomials are sparse. 
We remark that these results can be interpreted in another way: the number of algebraic integers (or algebraic numbers) of fixed degree and bounded height 
whose conjugates are multiplicatively dependent. 

In fact, the additive and multiplicative relations in conjugate algebraic numbers have been extensively studied. 
There are two typical problems. One is to detect whether there is an additive or multiplicative relation among the roots of a given irreducible polynomial, 
such as providing necessary or sufficient conditions; 
see, for instance, \cite{bds,Dixon,ds1,ds2,girst,kurb1,kurb2,kurb3,smy2} and more recently \cite{fl2,vall}. 
The other is to investigate which numbers can be represented by an additive or multiplicative relation in conjugate algebraic numbers 
for a fixed non-zero integer vector $(k_1,\dots,k_n)$; see \cite{dub,Du2003,DuSm,smy2}. 
In this paper, we want to investigate how often an additive or multiplicative relation of polynomial roots can occur among integer polynomials, i.e., to study these relations from the counting aspect. 
Our results suggest that except for some obvious cases these relations occur rarely. In some sense this corresponds to the results of Drmota
and Ska{\l}ba given in \cite{ds2}  (even though the normalization and the methods used are completely different): it was shown there that the multiplicative relations among conjugates of algebraic integers lying
in a fixed normal extension $F$ of $\Q$ occur rarely.

\subsection{The monic case}

From now on, suppose that $H \ge 3$ is a positive integer.  
Here and below, by $\# S$ we denote the number of elements of a finite set $S$. 
Also, for two functions $U=V(n,H)$ and $V=V(n,H)$ we will write $U \ll V$ or $U=O(V)$  if the inequality $|U| \leq c|V|$ holds for some positive constant $c$ depending on 
$n$ only but not on $H$ (except for the case when $\varepsilon$ appears, where the implied constant also depends on $\varepsilon$).  
Besides, $U \asymp V$ means that $U \ll V \ll U$. 

Let $M_n(H)$ be the set of generalized degenerate monic integer polynomials
of degree $n$ and height at most $H$. Evidently,
\begin{equation}\label{jons20}
\# M_n(H) \geq 2(2H+1)^{n-1} > 2^{n} H^{n-1},
\end{equation}
because each monic integer polynomial
of degree $n$ and height at most $H$ with constant coefficient $\pm 1$ belongs to the set $M_n(H)$.

To start with, by combining several results from different sources, we can obtain the following upper bound 
\begin{equation}\label{derfis}
\#M_n(H) \ll H^{n-1+\delta(n)+\eps}, 
\end{equation}
 where $\eps>0$, $n \geq 4$, $\delta(n)=1/n$ for $4 \leq n \leq 8$ and $\delta(n)=2/{n \choose \lfloor n/2 \rfloor}$ for $n \geq 9$.
 
Firstly, by the main result in \cite{Chela}, we have
\begin{equation}\label{poij1}
\#\{ f \in M_n(H) \>:\> f  \>\> \text{reducible in} \>\> \Z[X]\} \asymp H^{n-1}.
\end{equation}
Secondly, by \cite[Theorem 2]{dusa1}, 
\begin{equation}\label{poij2}
\#\{ f \in M_n(H) \>:\> f  \>\> \text{degenerate and irreducible in} \>\> \Z[X]\} \asymp H^{n/p},
\end{equation}
where $p$ is the smallest prime divisor of $n$. 
So, in view of \eqref{poij1}, \eqref{poij2} and the fact that the number of polynomials $f \in M_n(H)$ satisfying 
$f(0)=\pm 1$ does not exceed $2(2H+1)^{n-1}$, in order to prove \eqref{derfis} it remains to show that the bound \eqref{derfis} holds for
the polynomials $f \in M_{n}(H)$ which are irreducible, non-degenerate and satisfy $|f(0)| \geq 2$. 

Now, by the next lemma which is a consequence of \cite[Theorem 3]{bds}, 
we can further restrict this set to the set of polynomials whose Galois group is not $2$-transitive.  

\begin{lemma} 
\label{lem:mure}
Suppose that $\alpha$ is an algebraic number of degree $n \geq 2$ over
$\Q$ such that its conjugates $\alpha_1=\alpha,\alpha_2,\ldots,\alpha_n$
 are multiplicatively dependent, 
 and the Galois group of the field $\Q(\al_1,\dots,\al_n)$
over $\Q$ is $2$-transitive. Then, either $\alpha_1\cdots \alpha_n = \pm 1$ or for some positive integer $N$ we have $\al_1^N=\dots=\al_n^N$ (and so the minimal polynomial of $\al$ is
degenerate). 
\end{lemma}

Both the full symmetric group $S_n$ and the alternating group $A_n$ are $2$-transitive for $n \geq 4$.
Therefore, by Lemma~\ref{lem:mure} and 
the results of Dietmann \cite[Theorem 1 and Corollary 1]{Dietmann}, we derive the bound \eqref{derfis}. 
(Here, we also use a well-known fact that for a subgroup $G$ of $S_n$, if $G\ne S_n, A_n$, then $\# S_n / G \ge n$.)  

Let $L_n(H)$ be the set of monic integer polynomials of degree $n$ and height at most $H$ whose distinct roots are linearly dependent.
Then, applying similar arguments (note that \cite[Theorem 3]{bds} also gives a result on linear dependence), we obtain
\begin{equation}   \label{eq:linear1}
H^{n-1} \ll \# L_n(H) \ll H^{n-1+\delta(n) - \varepsilon}, \quad n \ge 4, 
\end{equation}
where the lower bound comes from those polynomials with zero coefficient for the term $X^{n-1}$. 

In this paper, we want to improve the bound \eqref{derfis}. More precisely, we will remove the factor $H^{\delta(n)}$ 
and replace the factor $H^{\varepsilon}$ by a logarithmic factor. 
Unfortunately, the same technique does not work for bounding the size of $L_n(H)$, so we do not know how to improve the upper bound in \eqref{eq:linear1} in general. We remark,
however, that some results below also hold for $L_n(H)$, especially when $n$ is a prime number.

Let  $I_n(H)$ (resp. $R_n(H)$) be the set of generalized degenerate monic irreducible (resp. reducible) integer polynomials
of degree $n$ and height at most $H$. 
Clearly, 
$$
\# M_n(H) = \# I_n(H) + \# R_n(H). 
$$  
We estimate $\# I_n(H)$ and  $\# R_n(H)$ separately. 

\begin{theorem}    \label{thm:irre}
For $\# I_n(H)$, we have:
\begin{itemize}
\item[(i)] for any integer $n\ge 2$, 
$$
H^{n-1} \ll \# I_n(H)  \ll H^{n-1} (\log H)^{2n^2-n-1}; 
$$

\item[(ii)] for any odd prime $p$, 
$$
\# I_p(H) = 2^p H^{p-1} + O(H^{p-2}), 
$$
and 
$$
\# I_2(H) = 6 H + O(H^{1/2}); 
$$

\item[(iii)]  $\# I_4(H)  \asymp H^3$. 
\end{itemize}
\end{theorem}

We even obtain an asymptotic formula for $\# R_n(H)$ for any $n \ge 2$ as $H \to \infty$.  
For this, we need to introduce some additional notation. 
For any $n\ge 2$, let $\nu_n$ be the volume of the symmetric convex body defined by 
$$
|x_i| \le 1, \, i=1,\ldots, n-1, \qquad \Big|\sum_{i=1}^{n-1} x_i\Big| \le 1. 
$$
Then, $\nu_2=2, \nu_3=3$ and $\nu_4=16/3$ (see, for instance, \cite[Section 5]{Du2014}). 

\begin{theorem}    \label{thm:reducible}
For $\# R_n(H)$, we have: 
\begin{itemize}
\item[(i)]   $\# R_2(H) = 4H + O(H^{1/2})$;

\item[(ii)]   $\# R_3(H) = 6 H^2 + O(H\log H \log\log H)$; 

\item[(iii)] for any $n\ge 4$, 
$$
\# R_n(H) = 2\nu_n H^{n-1} + O(H^{n-2}(\log H)^{2n^2-5n+2} ), 
$$
where the factor $(\log H)^{2n^2-5n+2}$ can be replaced by $\log H$ when $n-1$ is a prime or $n=5$. 
\end{itemize}
\end{theorem}

Combining Theorem \ref{thm:irre} with Theorem \ref{thm:reducible}, we immediately obtain 
the following estimates on the size of the set $M_n(H)$.  

\begin{theorem}\label{thm:monic}
For $\# M_n(H)$, we have: 
\begin{itemize}
\item[(i)] for any integer $n\ge 2$, 
$$
H^{n-1} \ll \# M_n(H)  \ll H^{n-1} (\log H)^{2n^2-n-1}; 
$$

\item[(ii)]   $\# M_2(H) = 10H + O(H^{1/2})$; 

\item[(iii)]   $\# M_3(H) = 14 H^2 + O(H\log H \log\log H)$; 

\item[(iv)] $\# M_4(H) \asymp H^3$; 

\item[(v)] $\# M_5(H) = (2\nu_5+32) H^4 + O(H^3\log H)$; 

\item[(vi)]  for any prime $p > 5$, 
$$
\# M_p(H) = (2\nu_p+2^p) H^{p-1} + O(H^{p-2}(\log H)^{2p^2-5p+2} ).
$$
\end{itemize}
\end{theorem}

 It seems very likely that the logarithmic factor
in Theorem \ref{thm:monic} (i) can be removed, since it is natural to expect that the growth rate $H^{n-1}$ is true not only for $n=4$ and prime, but also for each $n \ge 2$. 

\begin{conjecture}   \label{conj:monic}
For any integer $n\ge 2$, we have 
$$
 \# M_n(H) \asymp H^{n-1}. 
$$
\end{conjecture}

\subsection{The non-monic case}

Let $M_n^*(H)$ be the set of generalized degenerate integer polynomials (not necessarily monic)
of degree $n$ and height at most $H$. Obviously, we have 
$$
\# M_n^*(H) \geq 4H(2H+1)^{n-1} > 2^{n+1} H^n,
$$
since each integer polynomial
of degree $n$ and height at most $H$ with modulus of the constant coefficient equal to the modulus of the leading coefficient belongs to $M_n^*(H)$. 

As in the monic case, let  $I_n^*(H)$ (resp. $R_n^*(H)$) be the set of gene\-ra\-lized degenerate irreducible (resp. reducible) integer polynomials
of degree $n$ and height at most $H$. 
Clearly, 
$$
\# M_n^*(H) = \# I_n^*(H) + \# R_n^*(H). 
$$  
As before, we first estimate $\# I_n^*(H)$ and $\# R_n^*(H)$ separately. 

\begin{theorem}    \label{thm:irre2}
For $\# I_n^*(H)$, we have:
\begin{itemize}
\item[(i)] for any integer $n\ge 2$, 
$$
H^{n} \ll \# I_n^*(H)  \ll H^n (\log H)^{2n^2-n-1}; 
$$

\item[(ii)] for any odd prime $p$, 
$$
\# I_p^*(H) = 2^{p+1} H^p + O(H^{p-1}), 
$$
and 
$$
\# I_2^*(H) = 12 H^2 + O(H\log H);
$$

\item[(iii)]  $\# I_4^*(H)  \asymp H^4$. 
\end{itemize}
\end{theorem}

We also get an asymptotic formula for $\# R_n^*(H)$ for any $n \ge 2$ as $H \to \infty$.  

\begin{theorem}    \label{thm:reducible2}
For $\# R_n^*(H)$, we have: 
\begin{itemize}
\item[(i)]   $\# R_2^*(H) = 6H^2 + O(H\log H)$;

\item[(ii)] $\# R_3^*(H) = 2\nu_4 H^3 + O(H^2(\log H)^3)$; 

\item[(iii)] for any $n\ge 4$, 
$$
\# R_n^*(H) = 2\nu_{n+1} H^n + O(H^{n-1}(\log H)^{2n^2-5n+2} ), 
$$
where the factor $(\log H)^{2n^2-5n+2}$ can be replaced by $\log H$ when $n-1$ is a prime or $n=5$. 
\end{itemize}
\end{theorem}

Combining Theorem \ref{thm:irre2} with Theorem \ref{thm:reducible2}, we obtain 
the following estimates on the size of the set $M_n^*(H)$.  

\begin{theorem}\label{thm:nonmonic}
For $\# M_n^*(H)$, we have: 
\begin{itemize}
\item[(i)] for any integer $n\ge 2$, 
$$
H^{n} \ll \# M_n^*(H)  \ll H^{n} (\log H)^{2n^2-n-1}; 
$$

\item[(ii)]   $\# M_2^*(H) = 18H^2 + O(H\log H)$; 

\item[(iii)] $\# M_3^*(H) = (2\nu_4+16) H^3 + O(H^2(\log H)^3)$; 

\item[(iv)]  $\# M_4^*(H)  \asymp H^4$; 

\item[(v)] $\# M_5^*(H) = (2\nu_6+64) H^5 + O(H^4\log H)$; 

\item[(vi)]   for any prime $p > 5$,  
$$
\# M_p^*(H) = (2\nu_{p+1}+2^{p+1}) H^p + O(H^{p-1}(\log H)^{2p^2-5p+2} ). 
$$
\end{itemize}
\end{theorem}

We also conjecture that the logarithmic factor
in Theorem \ref{thm:nonmonic} (i) can be removed. 

\begin{conjecture}   \label{conj:nonmonic}
For any integer $n\ge 2$, we have 
$$
\# M_n^*(H) \asymp H^{n}. 
$$
\end{conjecture}

\section{Preliminaries}  

In this section, we gather some concepts and results used later on. 

\subsection{Basic concepts} 

Given a polynomial 
$$
f(X)=a_nX^n+a_{n-1}X^{n-1}+\cdots+a_0=a_n (X-\al_1)\cdots (X-\al_n) \in \C[X],
$$ 
where $a_n \ne 0$, its {\it height} is defined by $H(f)=\max_{0 \leq j \leq n} |a_j|$,
and its {\it Mahler measure} by
$$
M(f)=|a_n| \prod_{j=1}^n \max\{1,|\al_j|\}.
$$
For each $f \in \C[x]$ of degree $n$, these quantities are related by the following well-known inequality
$$
H(f) 2^{-n} \leq M(f) \leq H(f) \sqrt{n+1}; 
$$
see, for instance, \cite[(3.12)]{Waldschmidt2000}. So, for a fixed $n$, one has
\begin{equation}\label{eq:Mahler}
H(f) \ll M(f) \ll H(f).
\end{equation}
If $f$ can be factored as the product of two non-constant polynomials $g,h\in \C[X]$ (that is, $f=gh$), 
then, by definition of the Mahler measure, we have 
$$
M(f)=M(g)M(h).
$$ 
So, combined  with \eqref{eq:Mahler} this yields
\begin{equation}   \label{eq:height}
H(g)H(h) \ll H(f) \ll H(g)H(h). 
\end{equation} 

 For an algebraic number $\alpha$ of degree $n$ (over $\Q$), its Mahler measure $M(\alpha)$ is the Mahler measure of its minimal polynomial $f$ over $\Z$. 
 For the \emph{(Weil) absolute height} $H(\alpha)$ of $\alpha$, we have
\begin{equation*}
H(\alpha)=M(\alpha)^{1/n}.
\end{equation*}
With this notation, from \eqref{eq:Mahler} one gets 
\begin{equation}   \label{eq:MW}
H(f)^{1/n} \ll H(\alpha) \ll H(f)^{1/n}. 
\end{equation}

\subsection{Counting roots of polynomials}

We start with the following simple result on the number of integer 
vectors at which a multivariate polynomial is zero.

\begin{lemma} \label{polyno}
Let $f \in \Z[X_1,\dots,X_m]$ be a polynomial of degree $d \ge 1$. 
Then, the number of vectors $(x_1,\dots,x_m) \in \Z^m$, whose
coordinates satisfy $|x_j| \leq H$ for each $j=1,\dots,m$ and 
 $f(x_1,\dots,x_m)=0$,
does not exceed $dm(2H+1)^{m-1}$. 
\end{lemma}

\begin{proof}
We proceed the proof by induction on $m \ge 1$. 
The statement is obvious for $m=1$, since a univariate polynomial of degree $d \geq 1$ has at most $d$ roots. Assume that the assertion of the lemma is true for $m=k$. 
For $m=k+1$, we can write 
\begin{equation}\label{jons}
f(X_1,\dots,X_k,X_{k+1})=g_{0} +g_1 X_{k+1}+\dots +g_r X_{k+1}^r,
\end{equation} 
where 
$r \ge 1$, $g_0,\dots,g_r \in \Z[X_1,\dots,X_k]$ and $g_r$ is not zero identically. By our assumption, the number of vectors $(x_1,\dots,x_k) \in \Z^k$, where
$|x_j| \leq H$ for $j=1,\dots,k$, satisfying $g_r(x_1,\dots,x_k)=0$ does not exceed 
$$
k d_r (2H+1)^{k-1}, \quad \text{where} \quad d_r=\deg g_r.
$$
 Consequently, the number of vectors $(x_1,\dots,x_k,x_{k+1}) \in \Z^{k+1}$, where $(x_1,\dots,x_k)$ is one of the above vectors and $|x_{k+1}| \leq H$ does not exceed 
 $$
 k d_r (2H+1)^k.
 $$

For any other vector $(y_1,\dots,y_k) \in \Z^k$, where $|y_1|,\dots,|y_k| \leq H$, which is not equal to one of the above  vectors $(x_1,\dots,x_{k})$, the inequality
$g_r(y_1,\dots,y_k) \ne 0$ holds. Evidently, the number of such vectors is at most $(2H+1)^k$. In view of \eqref{jons} for each of them 
there are at most $r$ values of $y_{k+1}$ such that 
$f(y_1,\dots,y_k,y_{k+1})=0$, 
which yields the upper bound $r(2H+1)^{k}$.

Combining the above estimates, we see that the total number of vectors $(x_1,\dots,x_{k+1})\in \Z^{k+1}$,
where
$|x_1|, \dots,|x_{k+1}| \leq H$, at which $f$ given in \eqref{jons} vanishes, does not exceed 
$$
k d_r (2H+1)^k+r (2H+1)^k \leq (k+1)d(2H+1)^{k},
$$
since $d_r+r =\deg g_r+r \leq d$.
This completes the proof of the lemma. 
\end{proof}

We remark that in Lemma \ref{polyno} the growth rate $H^{m-1}$ is optimal, because $f$ may have a linear factor in $\Z[X]$. 
However, if $m\ge 2$ and $f$ is irreducible over $\overline{\Q}$ (the algebraic closure of $\Q$) of degree $d \ge 2$, 
then the bound can be sharpened to $O(H^{m-2+1/d+\varepsilon})$ for any $\varepsilon>0$ (the implied constant depends on $d,m,\varepsilon$); see \cite[Theorem A]{Pila}.

\subsection{Counting some special polynomials}

Let throughout $F_{n,k}^*(H)$ (resp. $F_{n,k}(H)$) be the set of integer polynomials (resp. monic integer polynomials)
of degree $n$ which are of height at most $H$ and are reducible over $\Z$ with a factor (not necessarily irreducible) of degree $k$,
$1 \le k \le n/2$. 

The following result follows from \cite{Waerden1936} in the monic case (see also \cite{Chela}) 
and from \cite{Kuba} in the non-monic case. (As indicated in
\cite{arduu} the result of \cite{Kuba} is misstated for $n=4$; see  \cite[Lemma 6]{arduu} for a correct version of monic and non-monic cases.). 
Here we give a simple proof. 

\begin{lemma}\label{lem:reducible}
For integers $n\ge 2$ and $k\ge 1$, where $1 \le k < n/2$, we have
$$
 \# F_{n,k}(H) \asymp H^{n-k}, \qquad 
\# F_{n,k}^*(H) \asymp H^{n+1-k}; 
$$
if $k=n/2$, we have 
$$
\# F_{n,k}(H) \asymp H^{n-k}\log H, \qquad 
\# F_{n,k}^*(H) \asymp H^{n+1-k}\log H.
$$
\end{lemma} 

\begin{proof}  
Suppose $f \in F_{n,k}(H)$.  Then, $f=gh$, where $g,h\in \Z[X]$ and $g$ is of degree $k$. 
By \eqref{eq:height}, we have $H(g)H(h) \ll H$, that is,  
$$
H(h) \ll H/H(g). 
$$ 
To count such polynomials $f$, we fix the height of $g$, say $a$. Then, the number of choices for $g$ is $O(a^{k-1})$, 
because at least one of coefficients of $g$ is equal to $\pm a$, and the number of choices of $h$ is $O((H/a)^{n-k})$. 
Thus, in view of \eqref{eq:height} we have 
$$
\# F_{n,k}(H) \asymp \sum_{a=1}^{H} a^{k-1} (H/a)^{n-k} = H^{n-k} \sum_{a=1}^{H} 1/a^{n+1-2k}. 
$$
So, for   $1 \le k < n/2$, we obtain
$$
 \# F_{n,k}(H) \asymp H^{n-k}; 
$$
in case $k=n/2$ we have 
$$
\# F_{n,k}(H) \asymp H^{n-k}\log H.
$$
Similarly, we can get the desired estimates for $\# F_{n,k}^{*}(H)$.
\end{proof}

We now count some special kinds of generalized degenerate integer polynomials.

\begin{lemma}   \label{lem:PnH}
Let $P_n(H)$ be the set of irreducible monic polynomials $f \in \Z[X]$  of degree $n\ge 2$ and height at most $H$ such that $f(0)=\pm 1$.  
Then, we have 
$$
\# P_n(H) = 2^nH^{n-1} + O(H^{n-2}). 
$$
\end{lemma}

\begin{proof}
The desired result is trivial for $n=2$, so below
we assume that $n \ge 3$. 

Let $f \in \Z[X]$ be of degree $n$  and height at most $H$ such that $f(0)=\pm 1$.  
Assume that $f$ is reducible over $\Z$ and can be written as $f=gh$ with $g,h\in \Z[X]$, where 
$$
g(X)=X^m+b_{m-1}X^{m-1}+ \cdots + b_1X+b_0,$$ $$h(X)=X^k+c_{k-1}X^{k-1}+ \cdots + c_1X +c_0. 
$$
Since $f(0)=\pm 1$, we must have $b_0=\pm 1$ and $c_0=\pm 1$. 
In view of \eqref{eq:height}, the number of such reducible polynomials $f$ is $O(H^{m-1}\cdot H^{k-1})=O(H^{n-2})$. 
This completes the proof. 
\end{proof}

\begin{lemma}   \label{lem:PnH*}
Let $P_n^*(H)$ be the set of irreducible polynomials $f \in \Z[X]$  of degree $n \ge 2$ and height at most $H$ such that the leading coefficient of $f$ is equal to $\pm f(0)$.  
Then, we have 
\begin{equation*}
\# P_n^*(H) = 2^{n+1}H^n + \left\{ \begin{array}{ll}
      O(H \log H) & \textrm{if $n=2$,}\\
           O(H^{n-1}) & \textrm{if $n\ge 3$}.
                 \end{array} \right.
\end{equation*}
\end{lemma}

\begin{proof}
We first consider the case $n=2$. 
It suffices to count all the integer polynomials of the form $a_0X^2+a_1X \pm a_0$, where $|a_0|, |a_1| \le H$,
which split into two linear factors. 
The two roots of such a polynomial must be
$m/k$ and $k/m$ (or $-k/m$) with coprime integers $k,m$. Here, without loss of generality, we 
can assume that  $m \ge k >0$.
Then, the polynomial is divisible by $(kX-m)(mX - k)$ or $(kX-m)(mX + k)$, so
it has the form $b(kX-m)(mX \pm k)$ for some non-zero integer $b$.
This yields $km \leq H/|b|$ and $m^2 \pm k^2 \leq H/|b|$.
Since $k^2 \le H/|b|$, we obtain $m^2 \leq 2H/|b|$.
So, $k \leq m \leq \sqrt{2H/|b|}$. It is known that the number
of such coprime pairs $(k,m)$ satisfying $k \leq m \leq \sqrt{2H/|b|}$ is asymptotic to $(6/\pi^2) 2H/|b|$ when $H/|b| \to \infty$. 
Thus, the number of such polynomials is at most 
$$
O \Big( \sum_{b=1}^{H} H/b  \Big) = O(H\log H), 
$$
which implies the desired result for $n=2$. 

Now, let $n \ge 3$. 
To obtain the desired result, it suffices to count reducible polynomials $f \in \Z[X]$  of 
degree $n$ and height at most $H$ such that the leading coefficient of $f$ is equal to $\pm f(0)$.
By Lemma \ref{lem:reducible}, we only need to consider such reducible polynomials having a linear factor. 
Suppose that 
$$
f(X)=a_nX^n + \cdots + a_1X + a_0 = (b_1X+b_0)(c_{n-1}X^{n-1}+\cdots + c_1X + c_0), 
$$
where all the coefficients are in $\Z$ and $b_1c_{n-1} \ne 0$. 
By assumption, $a_n=\pm a_0$. We also have $a_n=b_1c_{n-1}$ and $a_0=b_0c_0$. 
So, if we fix $a_0,b_0,b_1$, then $c_0,c_{n-1}$ are also fixed up to a sign. 
For any non-zero integer $m$, let $\cD(m)$ be the set of its positive divisors. 
Then, the number of such reducible polynomials $f$ having a linear factor is at most 
\begin{equation}   \label{eq:Pn*}
O \Big( \sum_{a_0=1}^{H}  \sum_{\substack{b_0,b_1 \in \cD(a_0) \\ b_1 \le b_0}} (H/b_0)^{n-2}  \Big). 
\end{equation}

Since $\lcm[b_0,b_1]$ divides $a_0$, it also does not exceed $H$. 
Noticing that the number of $a_0$ divisible by both $b_0$ and $b_1$ is at most 
$H/\lcm[b_0,b_1]$, we find that 
 the estimate \eqref{eq:Pn*} becomes
\begin{align*}
&    O \Big( H^{n-2} \sum_{\substack{ b_1 \le b_0 \le H \\ \lcm[b_0, b_1] \le H}} \frac{H}{\lcm[b_0,b_1]} \cdot \frac{1}{b_0^{n-2}}  \Big)  \\
 & \quad =  O \Big( H^{n-1} \sum_{b_0=1}^{H} \frac{1}{b_0^{n-1}} \sum_{b_1 \le b_0} \frac{\gcd(b_0,b_1)}{b_1}  \Big) ,
\end{align*}
which, by  letting  $d=\gcd(b_0,b_1) \in \cD(b_0)$ and then $b_1=cd$ for some $c \le b_0/d$, reduces to 
\begin{align*}
&   O \Big( H^{n-1} \sum_{b_0=1}^{H} \frac{1}{b_0^{n-1}} \sum_{d \in \cD(b_0)} \sum_{c\le b_0/d} \frac{d}{cd}  \Big)  \\
 & \quad = O \Big( H^{n-1}\sum_{b_0=1}^{H} \frac{\# \cD(b_0) \log b_0}{b_0^{n-1}} \Big) = O\big(H^{n-1}\big), 
\end{align*}
where we use $n\ge 3$ and  the fact that $\# \cD(b_0) \ll b_0^\varepsilon$ for any $\varepsilon>0$ and $b_0$ large enough. 
\end{proof}

We remark that the error term for the case $n \geq 3$ in Lemma \ref{lem:PnH*} is optimal. It sufffices to consider the polynomials divisible by $x-1$. It is also optimal for $n=2$.
Indeed, let us fix any positive integer $b$ in the range
$1 \leq b \leq H/2$. Consider coprime pairs $(k,m)$, $1\le k<m$, where
$m<\sqrt{H/(2b)}$. Asymptotically, There are $(3/\pi^2)H/b$ of them, and each pair
gives a different polynomial $b(kX-m)(mX-k)$ satisfying all the conditions of the lemma except the irreducibility. 
Summing over $b$ we get that the number of the required polynomials is at least $H\log H$ up to a multiplicative constant.

\begin{lemma}   \label{lem:QnH}
Let $Q_n(H)$ be the set of monic polynomials in $\Z[X]$  of degree $n \ge 2$ and height at most $H$ such that they are divisible by $X+1$ or $X-1$.  
Then, we have 
$$
\# Q_n(H) = 2\nu_n H^{n-1} + O(H^{n-2}). 
$$
\end{lemma}

\begin{proof}
Suppose $f \in \Z[X]$ satisfies $f(-1)=0$. Then, by changing the signs of the coefficients of $f$ corresponding to odd powers, 
one obtains a polynomial $g\in \Z[X]$ for which $g(1)=0$. The converse is also true. 
So, we only need to count (twice) the polynomials in $Q_n(H)$ which are divisible by $X-1$, since the number of monic polynomials divisible by both $X-1$ and $X+1$ is $O(H^{n-2})$.
With these assumptions, the proof follows along the same lines as the proof of \cite[Lemma 4]{Du2014} (about the non-monic case). 
\end{proof}

The following is a special case of \cite[Lemma 4]{Du2014}. 

\begin{lemma}   \label{lem:QnH*}
Let $Q_n^*(H)$ be the set of polynomials in $\Z[X]$  of degree $n \ge 2$ and height at most $H$ such that they are divisible by $X+1$ or $X-1$.  
Then, we have 
$$
\# Q_n^*(H) = 2\nu_{n+1} H^n + O(H^{n-1}). 
$$
\end{lemma}

\subsection{Multiplicative relations with conjugate algebraic num\-bers}

The multiplicative independence of the conjugates of prime 
degree was first established by Kurbatov (see \cite{kurb1}, \cite{kurb2}, \cite{kurb3}).  
The lemma below is given in \cite[Theorem 3]{ds1}.

\begin{lemma} 
\label{lem:mure0}
Let $p \geq 3$ be a prime number and let 
$$
f(X)=X^p+a_{p-1}X^{p-1}+\dots+a_1X+a_0 \in \Q[X]
$$
be an irreducible polynomial over $\Q$ such that 
$a_0 \ne \pm 1$ and $a_j \ne 0$ for at least one $j$ in the range
$1 \leq j \leq p-1$. Then, its roots are multiplicatively independent. 
\end{lemma}

The following lemma is \cite[Theorem 3]{Dixon} (see also \cite[Corollary 2]{smy2} or \cite[Theorem $4'$]{dub}):

\begin{lemma} 
\label{lem:mure11}
Let $n \geq 3$ be positive integer and let
$\al_1,\dots,\al_n$ be non-degenerate conjugate algebraic numbers of degree $n$
over $\Q$. Assume that $k_1,\dots,k_n \in \Z$ and $|k_1| \geq \sum_{j=2}^n |k_j|$. Then,
$\al_1^{k_1}\al_2^{k_2}\cdots \al_n^{k_n} \notin \Q$.
\end{lemma}

The following result of Loxton and van der Poorten shows that if some (not necessarily conjugate!) algebraic numbers $\alpha_1,\ldots,\alpha_n$ are multiplicatively dependent,
 then one can find a multiplicative dependence relation, where the exponents $k_i$, $i=1,\dots,n$, are  not too large; 
see, for example, \cite[Theorem 3]{Loxton} or \cite[Theorem~1]{Poorten}.

\begin{lemma} 
\label{lem:exponent}
Let $n\geq 2$, and let $\alpha_1,\ldots,\alpha_n$ be multiplicatively dependent non-zero algebraic numbers of degree at most $d$ and height at most $H$. 
Then, there are $k_1,\ldots,k_n \in \Z$, not all zero, and  a positive number $c$, which depends only on $n$ and $d$,  such that
$$
\alpha^{k_1}_1\cdots\alpha^{k_n}_n=1
$$
and
$$
\max_{1\leq i\leq n}|k_i| \le c(\log H)^{n-1}.
$$
\end{lemma}

Recall that $\Q^*$ is the set of non-zero rational numbers, which is a group under multiplication. 
Given a non-zero algebraic number $\alpha$, let  $\cN(\alpha)$ be the norm of $\alpha$ over $\Q$ (i.e., the product of all its conjugates), and 
let $\Gamma(\alpha)$ be the multiplicative group generated by all the conjugates of $\alpha$. 

Now, we shall prove the following:

\begin{lemma}    \label{lem:multQ}
Let $\alpha$ be a non-zero algebraic number. Then,
\begin{itemize}
\item[(i)]   $\Gamma(\alpha) \cap \Q^* = \{1\}$ if $\cN(\al)=1$ and $-1 \notin \Gamma(\alpha)$;

\item[(ii)]   $\Gamma(\alpha) \cap \Q^* = \{\pm 1\}$ if $|\cN(\al)|=1$ and $-1 \in \Gamma(\alpha)$;

\item[(iii)]   $\Gamma(\alpha) \cap \Q^* = \{g^m \,:\> m \in \Z\}$ if $|\cN(\al)| \ne 1$ and $-1 \notin \Gamma(\alpha)$;

\item[(iv)]   $\Gamma(\alpha) \cap \Q^* = \{\pm g^m \,:\> m \in \Z\}$ if $|\cN(\al)| \ne 1$ and $-1 \in \Gamma(\alpha)$.
\end{itemize}
Here, in parts {\rm (iii)} and {\rm (iv)}, $g \in \Gamma(\alpha) \cap \Q^*$ satisfies 
\begin{equation}\label{ggkk}
|g| = \min \{|\beta| :\,  \beta\in \Gamma(\alpha)\cap \Q^*, |\beta|>1\}. 
\end{equation}
\end{lemma}

Note that the case $\cN(\alpha)=-1$ and $-1 \notin \Gamma(\alpha)$
in (i) is impossible, since $\cN(\alpha) \in \Gamma(\alpha) \cap \Q^*$.

\begin{proof} 
Assume that the algebraic number $\alpha$ is of degree $n$ over $\Q$ with conjugates $\alpha_1=\alpha, \alpha_2, \ldots, \alpha_n$. 
Denote $d=[\Q(\alpha_1,\ldots,\alpha_n):\Q]$.  
Let $g_0 \geq 1$ be the minimal positive rational number such that $|\cN(\alpha)|=g_0^a$ for some non-zero integer $a$. 
Then, $g_0=1$ if and only if $|\cN(\alpha)|=1$. Otherwise, there exist (pairwise distinct) prime numbers $p_1,\ldots,p_m$  and non-zero integers $r_1,\ldots,r_m$ such that 
$$
g_0=p_1^{r_1} \cdots p_m^{r_m}. 
$$

For a non-zero integer vector $\bk=(k_1,\ldots,k_n) \in \Z^n$, assume that 
$$
\beta(\bk) = \alpha_1^{k_1} \cdots \alpha_n^{k_n} \in \Gamma(\alpha) \cap \Q^*.
$$
Then, applying all $d$ automorphisms of $\Q(\alpha_1,\ldots,\alpha_n)$ over $\Q$ to $\beta(\bk)$ and multiplying all the obtained equalities, we deduce that 
$$
\beta(\bk)^{d} = (\alpha_1 \cdots \alpha_n)^{d(k_1+\cdots +k_n)/n} = \cN(\alpha)^{d(k_1+\cdots +k_n)/n} = \pm g_0^{ad(k_1+\cdots +k_n)/n}. 
$$
Now, $g_0=1$ implies that $\beta(\bk)=\pm 1$. This proves parts (i)
and (ii). 

From now on assume that $|\cN(\alpha)| \ne 1$, and so $g_0>1$. Suppose also that $\beta(\bk) \ne \pm 1$. Evidently,
we can assume that $a(k_1+\cdots +k_n)>0$, because otherwise we can replace $\beta(\bk)$ by $\beta(\bk)^{-1}$. 
Let us denote 
$$
q=ad(k_1+\cdots +k_n)/n,
$$
 which is in fact a positive integer (because $n \mid d$).  
Then, for the prime numbers $p_1,\ldots,p_m$ defined above and some non-zero integers $s_1,\ldots,s_m$ we must have 
$$
|\beta(\bk)| = p_1^{s_1} \cdots p_m^{s_m}, 
$$
and therefore
$$
|\beta(\bk)|^{d}=p_1^{ds_1} \cdots p_m^{ds_m} = p_1^{qr_1} \cdots p_m^{qr_m}=g_0^q.
$$
In particular, we obtain
\begin{equation}   \label{eq:rs}
ds_i = qr_i, \quad i=1,2,\ldots, m. 
\end{equation}
For each $1\le i \le m$, let $t_i = \gcd(r_i,s_i)$ if $r_i>0$ (that is, $s_i>0$), and $t_i = -\gcd(r_i,s_i)$ otherwise.
So, we always have $r_i/t_i>0$ and $s_i/t_i>0$. 

Thus, for each $1\le i \le m$, by \eqref{eq:rs} we deduce that
$$
\frac{r_i}{t_i} = \frac{|r_i|}{\gcd(r_i,s_i)} = \frac{|dr_i|}{\gcd(dr_i,ds_i)} = \frac{|dr_i|}{\gcd(dr_i,qr_i)} = \frac{d}{\gcd(d,q)}, 
$$
and similarly 
$$
\frac{s_i}{t_i} = \frac{q}{\gcd(d,q)}. 
$$
Therefore, putting $u=p_1^{t_1}\cdots p_m^{t_m}$ we obtain
$$
|\beta(\bk)| = u^{q/\gcd(d,q)}=g_0^{q/d}.$$  
This yields
$g_0 = u^{d/\gcd(d,q)}$.
In view of the choice of $g_0$, we must have $g_0=u$, and thus 
$$
|\beta(\bk)| = g_0^{q/\gcd(d,q)}. 
$$
So, we have proved that for any $\beta \in \Gamma(\alpha) \cap \Q^*$
which is not $\pm 1$, we have 
$|\beta| = g_0^b$ for some integer $b$. In particular this yields that
\begin{equation}\label{jins2}
\Gamma(\alpha) \cap \Q^* \subseteq \{\pm g_0^m \, :\> m \in \Z\}.
\end{equation}

Let $g \in \Gamma(\alpha) \cap \Q^*$ be defined by \eqref{ggkk}. We claim that
\begin{equation}\label{jins1}
\{g^m \>:\> m \in \Z\>\}  \subseteq \Gamma(\alpha) \cap \Q^*
\subseteq  \{\pm g^m \>:\>m \in \Z\>\}. 
\end{equation}
Indeed, it suffices to show that for any $\beta \in \Gamma(\alpha) \cap \Q^*$ 
its modulus $|\beta|$ is an integer power of $|g|$. This is clear for $\beta = \pm 1$.
Suppose that $|\beta| \ne 1$. Then, in view of \eqref{jins2} we have $g=\pm g_0^b$ for some positive integer $b$. 
Suppose that there exists $\beta \in \Gamma(\alpha) \cap \Q^*$ such that $\beta \ne \pm 1$ and $|\be|$ is not an integer power of $|g|$. 
Still, we have $\beta=\pm g_0^c$ for some integer $c \ne 0$. 
We can assume that $c>0$, otherwise we replace $\beta$ by $\beta^{-1}$.
By the division algorithm, we write $c=bw+v$ for some integers $w,v \geq 0$. 
By the choices of $g$ and $\beta$, we must have $w>0$ and $0< v < b$. 
Then, 
$
 \pm g_0^v = \pm g_0^{c-bw} = \pm \be/g^{w}, 
$ 
so either $g_0^v$ or $-g_0^v$ is in $\Gamma(\alpha)$. 
However, $1<g_0^v < g_0^b=|g|$, which contradicts the choice of 
$g$.
This completes the proof of \eqref{jins1}.  

To complete the proof of (iii) we assume that $- 1 \notin \Gamma(\alpha)$. Then, there does not exist $m \in \Z$ such that both $g^m$
and $-g^m$ are in $\Gamma(\alpha)$, since otherwise their quotient $-1$ is in $\Gamma(\alpha)$. This, in view of the left inclusion in \eqref{jins1} completes the proof of (iii).

Evidently, in case $-1 \in \Gamma(\alpha)$ in view of $g^m \in \Gamma(\alpha)$ we also have $-g^m \in \Gamma(\alpha)$. By \eqref{jins1}, this completes the proof of (iv). 
\end{proof}

Recall that $\N$ is the set of natural numbers (that is, positive integers). 
Clearly, 
by $\cN(\alpha) \in \Gamma(\alpha) \cap \Q^*$,
Lemma~\ref{lem:multQ} (iii), (iv)
 and \eqref{ggkk}, we must have $|g|^k=|\cN(\alpha)|$ for some $k \in \N$. Hence,
Lemma~\ref{lem:multQ} implies the following:

\begin{corollary}    \label{lem:multN}
Let $\alpha$ be a non-zero algebraic number.  
Then, we have $\Gamma(\alpha) \cap \N \ne  \{1\}$ if and only if either $\cN(\alpha)$ or $1/\cN(\alpha)$ is an integer not equal to $\pm 1$. 
\end{corollary}

\section{Proofs for the monic case}

\subsection{Proof of Theorem \ref{thm:irre}}   \label{sec:irre}

(i) 
Note that, by Lemma \ref{lem:PnH}, for any $n \ge 2$ we have 
$$
\# I_n(H) \gg H^{n-1}. 
$$
In all what follows we will prove the upper bound. 

Let $\al_1,\dots,\al_n$
be the roots of the monic irreducible polynomial 
$$
f=X^n+a_{n-1}X^{n-1}+\cdots +a_0 \in \Z[X],
$$ 
where $|a_0|,\dots,|a_{n-1}| \leq H$.  
Set
\begin{equation}\label{jons11}
K= \lfloor c (\log H)^{n-1}\rfloor, 
\end{equation}
where $c$ is the constant defined in Lemma \ref{lem:exponent} (with $d=n!$). Note that the height $H(\alpha_i)$ of each $\alpha_i$ is $O(H^{1/n})$, by \eqref{eq:MW}. 
Hence, $c$ depends only on $n$. 

Let ${\mathcal K}$ be the set of all the integer vectors $\bk =(k_1,\dots,k_n)$ with $0 \leq k_i \leq K$, where $i=1,\ldots, n$. 
Clearly, $\#{\mathcal K}=(K+1)^n$. 
For each $\bk=(k_1,\dots,k_n)\in {\mathcal K}$, we denote 
$$
\beta(\bk)=\al_1^{k_1} \cdots \al_{n}^{k_n}. 
$$ 

By Lemma~\ref{lem:exponent}, we see that $f\in I_n(H)$ (that is, the algebraic integers $\al_1,\dots,\al_n$ are multiplicatively dependent) 
if and only if $\be({\bf k}) =\be({\bf k}')$ for some ${\bf k} \ne {\bf k}' \in {\mathcal K}$. 

Now, consider the monic  polynomial (in $X$)
$$
F(X) = \prod_{{\bf k} \in {\mathcal K}} (X-\be({\bf k})) \in \Z[X]
$$
of degree $\#{\mathcal K}=(K+1)^{n}$. 
Its discriminant 
\begin{equation}\label{jons1}
\Delta(a_0,\dots,a_{n-1})=\pm \prod_{{\bf k} \ne {\bf k'} \in {\mathcal K}} (\be({\bf k})-\be({\bf k'}))
\end{equation}
 is a polynomial in $\Z[a_0,\dots,a_{n-1}]$, because it is a symmetric polynomial in $\alpha_1, \ldots, \alpha_n$.  
Note that the coefficients of $F$ are symmetric polynomials in terms of $\alpha_1, \ldots, \alpha_n$ with total degree at most $nK(K+1)^n/2$. 
By the fundamental theorem for symmetric polynomials and viewing $a_0,\ldots,a_{n-1}$ as elementary symmetric polynomials in $\alpha_1, \ldots, \alpha_n$, 
the coefficients of $F$ are also polynomials in $a_0,\ldots,a_{n-1}$ of total degree at most $nK(K+1)^n/2$. 

It is a well-known fact that $\Delta(a_0,\dots,a_{n-1})$ is a homogeneous polynomial in the coefficients of $F$ of total degree $2\deg F - 2$ (that is, $2(K+1)^n-2$). 
Hence, as a polynomial in $a_0,\ldots,a_{n-1}$, the total degree of the polynomial
$\Delta(a_0,\dots,a_{n-1})$ is at most 
\begin{equation}   \label{eq:degD}
(2(K+1)^n-2) \cdot nK(K+1)^n/2 < n(K+1)^{2n+1}.
\end{equation}

Notice that $f \in I_n(H)$ if and only if $\Delta(a_0,\dots,a_{n-1})=0$. 
Now, by Lemma~\ref{polyno}, \eqref{eq:degD} and the definition of $K$ in \eqref{jons11}, it follows that the number of vectors
$(a_0,\dots,a_{n-1}) \in \Z^n$ satisfying $|a_0|,\dots,|a_{n-1}| \leq H$
and $\Delta(a_0,\dots,a_{n-1})=0$, does not exceed 
\begin{equation}  \label{eq:general}
\begin{split}
(2H+1)^{n-1} n \deg \Delta & <
n^2 (K+1)^{2n+1} (2H+1)^{n-1} \\
& \le n^2 (c (\log H)^{n-1}+1)^{2n+1}(2H+1)^{n-1} \\
& \ll H^{n-1} (\log H)^{2n^2-n-1}. 
\end{split}
\end{equation}
This completes the proof of (i).

(ii)
For any odd prime $p$ the claimed 
asymptotic formula  follows directly from Lemmas \ref{lem:PnH} and \ref{lem:mure0}. 

For $n=2$, let  $f=X^2 + a_1X +a_0 \in \Z[X]$, $a_0\ne \pm 1$, be an irreducible polynomial, with roots $\alpha_1$ and $\alpha_2$. 
Assume that $\al_1^{k_1} \al_2^{k_2}=1$ for some integers $k_1,k_2$, not both zero. 
Applying the Galois automorphism which swaps $\al_1$ with $\al_2$ to the above multiplicative relation, we obtain
$$
(\alpha_1\alpha_2)^{k_1+k_2} = a_0^{k_1+k_2} =1. 
$$ 
Since $a_0 \ne \pm 1$, we must have $k_1+k_2 = 0$. 
So, $f$ is in fact degenerate. 
Then, the desired asymptotic formula for $n=2$ follows from Lemma \ref{lem:PnH} and \cite[Theorem 7]{dusa1}.

(iii) 
We finally consider the case $n=4$.  
Let $\alpha_1,\alpha_2,\alpha_3,\alpha_4$ be the roots  of a monic irreducible polynomial $f \in \Z[X]$. 
Assume that 
\begin{equation}\label{mkij1}
\alpha_1^{k_1} \alpha_2^{k_2} \alpha_3^{k_3} \alpha_4^{k_4}=1
\end{equation}
 with some $k_1,k_2,k_3,k_4 \in \Z$, not all zero (that is, $f\in I_n(H)$). 
As explained above in order
to prove the asymptotic formula,  we can assume that $f(0) \ne \pm 1$ and that $f$ is non-degenerate.

First, applying all the automorphisms in the Galois group 
$$
G={\rm Gal}(\Q(\alpha_1,\alpha_2,\alpha_3,\alpha_4)/\Q)
$$ 
to \eqref{mkij1} and then multiplying all the obtained equalities we get
$$
\cN(\alpha_1)^{\#G(k_1+k_2+k_3+k_4)/4}=1.
$$
Hence, in view of $\cN(\alpha_1)=f(0) \ne \pm 1$ we deduce that 
\begin{equation}\label{mkij10}
k_1+k_2+k_3+k_4=0.
\end{equation}

Suppose now that there are exactly two $k_i$ not equal to zero. Then, without loss of generality we may assume that
$k_1 \ne 0$, $k_2=-k_1$, and $k_3=k_4=0$. However, $\alpha_1^{k_1}\alpha_2^{-k_1}=1$ means that $f$ is degenerate, which is not the case. 

Next, assume that exactly three $k_i$ are non-zero, say $k_1,k_2,k_3 \ne 0$. Then, in view of $k_1+k_2+k_3=0$, the modulus of the largest $|k_i|$, $i=1,2,3$, equals the sum of the moduli of the other two, for example, $|k_1|=|k_2|+|k_3|$.  However, for such $k_i$, since $f$ is non-degenerate, the equality $\alpha_1^{k_1}\alpha_2^{k_2}\alpha_3^{k_3}=1$ is impossible, by Lemma~\ref{lem:mure11}.

Finally, assume that all four $k_i$, $i=1,2,3,4$, are non-zero. 
By the same argument as above, we cannot have
$|k_i|=\sum_{j \ne i}|k_j|$, so 
exactly two of the four $k_i$'s are positive and the other two are negative. 
Without restriction of generality we may assume that $0<k_1 \ne -k_2$.  
Take an automorphism $\sigma \in G$ that maps $\al_2 \mapsto \al_1$.
Putting $\alpha_i=\sigma(\alpha_1)$, $\alpha_j=\sigma(\alpha_3)$ and $\alpha_k=\sigma(\alpha_4)$, where $\{i,j,k\}=\{2,3,4\}$, we derive that
$\alpha_1^{k_2} \alpha_i^{k_1} \alpha_j^{k_3} \alpha_k^{k_4}=1$. Combining this with \eqref{mkij1} we obtain 
\begin{equation}\label{blyns}
\alpha_2^{k_2^2} \alpha_3^{k_2k_3}\alpha_4^{k_2k_4} =\alpha_1^{-k_1k_2} = \alpha_i^{k_1^2}\alpha_j^{k_1k_3}\alpha_k^{k_1k_4}.
\end{equation}
Here, in view of $\{i,j,k\}=\{2,3,4\}$ and \eqref{mkij10}, the sums of the exponents in each side of \eqref{blyns} are all equal to $-k_1k_2$. 
Now, dividing the left hand side of \eqref{blyns} by its right hand side, we obtain 
another multiplicative relation 
$$
\alpha_2^{q_2}\alpha_3^{q_3}\alpha_4^{q_4}=1,
$$
where $q_2,q_3,q_4 \in \Z$ are such that $q_2+q_3+q_4=0$. 
However, as we have already showed above, this is impossible if at least one $q_j$, $j=2,3,4$, is non-zero. 

It remains to consider the case when $q_2=q_3=q_4=0$.
This happens precisely when the list of exponents $k_2^2,k_2k_3,k_2k_4$ on the left hand side of \eqref{blyns}
is a permutation of the list  of exponents on the right hand side $k_1^2,k_1k_3,k_1k_4$. 
In particular, this implies that the products of those exponents, i.e., $k_2^4k_3k_4$ and $k_1^4k_3k_4$, respectively, must be equal. It follows that $k_1^4=k_2^4$, and hence 
$k_1 = k_2$ (because we have assumed $k_1 \ne -k_2$).  

By the same argument, for any pair $(i,j)$, where $1 \leq i<j \leq 4$,
we have either $k_i=-k_j$ or $k_i=k_j$.   
Hence, $k_1^2=k_2^2=k_3^2=k_4^2$. Consequently, without restriction of generality we may assume that
$k_1=k_2=k>0$ and $k_3=k_4=-k$. 
Then, \eqref{mkij1} reduces to
\begin{equation}\label{kimkim1}
\alpha_1^k\alpha_2^k\alpha_3^{-k}\alpha_4^{-k} = 1, 
\end{equation}
where $k$ is a positive integer. 
This yields
\begin{equation}\label{kimkim}
\alpha_3\alpha_4=\zeta \alpha_1 \alpha_2, 
\end{equation}
where $\zeta$ is a root of unity. 

Since the polynomial $f$ is non-degenerate, from \eqref{kimkim} it is straightforward to see that
any automorphism $\sigma \in G$ must map the set $\{\alpha_1,\alpha_2\}$ either to itself or the set $\{\alpha_3,\alpha_4\}$.  
So, the conjugates of $\alpha_1\alpha_2$ are $\alpha_1\alpha_2$ and $\alpha_3\alpha_4$. 
That is, the field $\Q(\alpha_1\alpha_2)=\Q(\alpha_3\alpha_4)$ is of at most degree two over $\Q$. 
Hence, $\zeta$ is of  degree at most two over $\Q$. 
Thus, in \eqref{kimkim1} we can always choose $k=12$. 

Then, for $n=4$ we can select in \eqref{jons11} the absolute constant $K=12$ (instead of $K= \lfloor c (\log H)^3\rfloor$). 
So, with $n=4$ and $K=12$ in \eqref{eq:general} we get the upper bound $O(H^3)$ for $\# I_4(H)$.

\subsection{Proof of Theorem \ref{thm:reducible}}

(i) 
First, it is easy to see that the number of polynomials in $R_2(H)$ which are divisible by $X+1$ or $X-1$ 
is equal to $4H + O(1)$. 

Note that according to our definition, the polynomials of the form $X(X+b)$, where
$b \in \Z \setminus \{\pm 1\}$, are not contained in $R_2(H)$.   
For polynomials in $R_2(H)$ of the form $(X+b_0)(X+b_1)$, where $b_0,b_1\in \Z$ are not equal to $\pm 1$ and $0$, 
since $b_0$ and $b_1$ are multiplicatively dependent, there exists a positive integer $a>1$ such that $|b_0|=a^k, |b_1|=a^m$ 
for some positive integers $k,m$ with $k+m \le \log_a H$. Since $k+m \ge 2$, we must have $a\le \sqrt{H}$. 
Then, in view of $\sum_{k =2}^{T} 1/(\log k)^2=O(T/(\log T)^2)$, the number of these polynomials is at most 
$$
O\Big( \sum_{a=2}^{\sqrt{H}} \sum_{k=1}^{\log_a H} ( \log_a H  -k)  \Big) = O\Big( (\log H)^2 \sum_{a=2}^{\sqrt{H}} \frac{1}{(\log a)^2}\Big)=O\Big(\sqrt{H}\Big). 
$$
Hence, we obtain 
$$
\# R_2(H) = 4H + O(H^{1/2}),
$$
as claimed.

(ii)
By Lemma \ref{lem:QnH}, the number of polynomials in $R_3(H)$ which are divisible by $X+1$ or $X-1$ is equal to 
$$
2\nu_3 H^2 + O(H). 
$$

For any polynomial $f\in R_3(H)$ of the form $(X+b_0)(X + b_1)(X+b_2)$ with $b_0,b_1,b_2\in \Z$ not equal to $\pm 1$ and $0$, 
since $b_0,b_1,b_2$ are multiplicatively dependent, for each prime factor $p$ of $b_0b_1b_2$ we have $p^2 \mid b_0b_1b_2$.   
Let $a$ be the positive integer such that $a^2$ is the maximal square divisor of $|b_0b_1b_2|$. 
So, $|b_0b_1b_2|$ is a divisor of $a^3$. 
Observing that $|b_0b_1b_2| \le H$, we have $a \le \sqrt{H}$. 
Recall that for a non-zero integer $m$, $\cD(m)$ is the set of its positive divisors.  
Then,  we see that the number of such polynomials $f$ is at most 
\begin{align*}
O \Big( \sum_{a=1}^{\sqrt{H}} \sum_{b_0,b_1,b_2 \in \cD(a^3)} 1 \Big) 
&= O \Big( \sum_{a=1}^{\sqrt{H}} \# \cD(a^3)^3 \Big)
= O \Big( \sum_{a=1}^{\sqrt{H}} \# \cD(a)^9 \Big) \\&=O(H^{1/2}(\log H)^{511}), 
\end{align*}
where we used the bound
\begin{equation}   \label{eq:PnPn}
\sum_{k \leq T} \# \cD(k)^s \asymp T (\log T)^{2^s-1}
\end{equation}
for $s=9$ due to Wilson \cite{Wilson} (for a generalization see \cite{Delmer} and also \cite{Luca}).

For any polynomial $f\in R_3(H)$ of the form $(X+b_0)(X^2+c_1X+c_0)$, where $b_0 \ne \pm 1$ and $X^2+c_1X+c_0$ irreducible, 
either the roots of $X^2+c_1X+c_0$ are multiplicatively dependent, or $b_0$ and $c_0$  (where $b_0, c_0 \notin \{0,\pm 1\}$) are multiplicatively dependent, by Lemma \ref{lem:multQ}. 
In the first case, by Theorem \ref{thm:irre} (ii),  the number of such polynomials $f$ is at most 
$$
O(\sum_{b_0=2}^{H} H/b_0) = O(H\log H). 
$$
In the second case, there exists a positive integer $a>1$ such that $|b_0|=a^k, |c_0|=a^m$ 
for some positive integers $k,m$ with $k+m \le \log_a H$, where $a \le \sqrt{H}$ in view of $a^2 \le a^{k+m} =|b_0c_0| \le H$.
Clearly, the number of such polynomials $f$ is at most 
\begin{align*}
O\Big( \sum_{a=2}^{\sqrt{H}} \sum_{k=1}^{\log_a H} (\log_a H -k) H/a^k \Big) 
& = O\Big(H \log H \sum_{a=2}^{\sqrt{H}}1 /(a\log a) \Big) \\
& = O(H\log H \log\log H). 
\end{align*}

Collecting the above estimates, we obtain 
$$
\# R_3(H) = 2\nu_3 H^2 + O(H\log H \log\log H). 
$$

(iii) 
Now, assume that $n\ge 4$. 
By Lemma \ref{lem:reducible}, we only need to consider polynomials in $R_n(H)$ which have a linear factor. 
First, by Lemma \ref{lem:QnH} the number of polynomials in $R_n(H)$ which are divisible by $X+1$ or $X-1$ is equal to 
$$
2\nu_n H^{n-1} + O(H^{n-2}). 
$$

For any polynomial $f\in R_n(H)$ of the form $(X+b)g$ with irreducible $g\in \Z[X]$ and $b \ne \pm 1$, we have that 
either the roots of $g$ are multiplicatively dependent, or $b$ and $g(0)$  (both $b$ and $g(0)$ are not equal to $0, \pm 1$) are multiplicatively dependent, by Lemma \ref{lem:multQ}. 
In the first case, by Theorem \ref{thm:irre} (i), the number of such polynomials $f$ is at most 
$$
O\Big( \sum_{b=2}^{H} (H/b)^{n-2} (\log(H/b))^{2(n-1)^2-(n-1)-1} \Big) = O(H^{n-2}(\log H)^{2n^2-5n+2} ), 
$$ 
where the factor $(\log H)^{2n^2-5n+2}$ can be removed when $n-1$ is a prime or $n=5$ (see Theorem \ref{thm:irre} (ii) and (iii)). 
In the second case, since there exists a positive integer $a>1$ such that $|b|=a^k, |g(0)|=a^m$ 
for some positive integers $k,m$ with $k+m \le \log_a H$, the number of such polynomials $f$ is at most 
\begin{align*}
O\Big( \sum_{a=2}^{\sqrt{H}} \sum_{k=1}^{\log_a H} (\log_a H -k) (H/a^k)^{n-2} \Big) 
& = O\Big(H^{n-2} \log H \sum_{a=2}^{\sqrt{H}} a^{2-n} \Big) 
 \\&= O(H^{n-2}\log H). 
\end{align*}

Collecting the above estimates, we obtain
$$
\# R_n(H) = 2\nu_n H^{n-1} + O(H^{n-2}(\log H)^{2n^2-5n+2} ), 
$$
where the factor $(\log H)^{2n^2-5n+2}$ can be replaced by $\log H$ when $n-1$ is a prime or $n=5$.

\section{Proofs for the non-monic case}

\subsection{Proof of Theorem \ref{thm:irre2}} 

(i)
Firstly, note that the set $P_n^*(H)$ defined in Lemma \ref{lem:PnH*} is contained in $I_n^*(H)$. 
So, it suffices to show the upper bound. 
For a polynomial 
$$
f=a_nX^n+a_{n-1}X^{n-1}+\cdots +a_0 \in I_n^*(H),
$$ 
let 
$$
g = X^n + \frac{a_{n-1}}{a_n} X^{n-1} + \cdots + \frac{a_0}{a_n}. 
$$
Then, applying the same arguments as in Section \ref{sec:irre} to the polynomial $g$ (with the same notation as there), 
we deduce that the total degree of the polynomial
$\Delta(a_0/a_n,\dots,a_{n-1}/a_n)$ (in the variables $a_0/a_n,\dots,a_{n-1}/a_n$) is also at most 
\begin{equation*}   
(2(K+1)^n-2) \cdot nK(K+1)^n/2 < n(K+1)^{2n+1}.
\end{equation*}
So, as a polynomial in $a_0,\ldots,a_{n-1},a_n$, the total degree of the polynomial
$a_n^{n(K+1)^{2n+1}}\Delta(a_0/a_n,\dots,a_{n-1}/a_n)$ is at most $2n(K+1)^{2n+1}$.

Note that $f \in I_n^*(H)$ if and only if 
$$
a_n^{n(K+1)^{2n+1}}\Delta(a_0/a_n,\dots,a_{n-1}/a_n)=0.
$$
Now, by Lemma~\ref{polyno},  it follows that the number of polynomials in $I_n^*(H)$ is at most 
\begin{equation}   \label{eq:general222}
(n+1) \cdot  2n(K+1)^{2n+1} (2H+1)^n.
\end{equation}
By the definition of $K$ in \eqref{jons11}, this is not greater than 
\begin{equation}   \label{eq:general2}
2(n+1)^2(c (\log H)^{n-1}+1)^{2n+1}(2H+1)^n \ll H^n (\log H)^{2n^2-n-1}. 
\end{equation}
This completes the proof of (i).

(ii)
For an odd prime $p$, the claimed 
asymptotic formula  follows directly from  Lemmas \ref{lem:PnH*} and \ref{lem:mure0}. 

Next, consider an irreducible polynomial $f=a_2X^2 + a_1X +a_0 \in I_2^*(H)$ with $a_0\ne \pm a_2$. 
As in the monic case, $f$ is in fact degenerate. 
Let $\alpha_1,\alpha_2$ be the roots of $f$. 
Since $\alpha_1/\alpha_2$ is of at most degree two over $\Q$,  we must have $\alpha_1/\alpha_2 =-1, \pm \sqrt{-1}, (1\pm \sqrt{-3})/2$ 
or $(-1 \pm \sqrt{-3})/2$.  

If $\alpha_1/\alpha_2 =-1$, we have $\alpha_1+\alpha_2=0$, and so $a_1=0$. 
So, $f$ is of the form $a_2X^2+a_0$. 
Since $f$ is irreducible, we exactly need to exclude polynomials of the form $ca^2X^2-cb^2$ with $a,b,c\in \Z$. 
It is easy to see that the number of such polynomials $ca^2X^2-cb^2$ with $|ca^2| \le H, |cb^2| \le H$ is at most 
$$
O \Big( \sum_{c=1}^{H}(H/c)^{1/2} \cdot (H/c)^{1/2} \Big) = O(H\log H). 
$$
So, the number of polynomials $f\in I_2^*(H)$ of the form $a_2X^2 + a_0$ is equal to 
$$
4H^2 + O(H\log H). 
$$

In case $\alpha_1/\alpha_2 = \pm \sqrt{-1}$ we find that $\alpha_1^2+\alpha_2^2=0$. 
Observing that $\alpha_1+\alpha_2=-a_1/a_2$ and $\alpha_1\alpha_2=a_0/a_2$,  we obtain 
$$
\alpha_1^2+\alpha_2^2 = (a_1/a_2)^2 - 2a_0/a_2 = 0. 
$$
Hence, $a_1^2 = 2a_0a_2$. 
Let $b$ be the positive integer such that $b^2$ is the maximal square divisor of $2|a_0|$. 
We write $2|a_0|=b^2c$, where $c$ is square-free. 
Since $a_1^2 = b^2c|a_2|$, the integer $c|a_2|$ must be a perfect square. As $c$ divides $|a_2|$, the integer 
$|a_2|/c$ is also a perfect square. 
So, we can write $|a_2| = c d^2$, where $d \le \sqrt{H/c}$. 
Since $|a_0| \le H$, we have $b \le \sqrt{2H}$. 
If we fix $b,c,d$, then $a_0,a_1,a_2$ are also fixed up to a sign. 
Thus, the number of corresponding polynomials $f$ is at most 
$$
O \Big( \sum_{b=1}^{\sqrt{2H}} \sum_{c=1}^{2H/b^2} \sqrt{H/c} \Big) 
=O \Big( H\sum_{b=1}^{\sqrt{2H}} 1/b \Big)= O(H\log H).  
$$

For $\alpha_1/\alpha_2 = (1\pm \sqrt{-3})/2$, we find that $(\alpha_1/\alpha_2)^2 - \alpha_1/\alpha_2 + 1 =0$. 
So, we get 
$$
0 = \alpha_1^2 - \alpha_1\alpha_2 + \alpha_2^2 = (a_1/a_2)^2 - 3a_0/a_2, 
$$
which implies that $a_1^2 = 3a_0a_2$. 
As the above, the number of corresponding polynomials $f$ is at most $O(H\log H)$. 

Similarly, if $\alpha_1/\alpha_2 = (-1\pm \sqrt{-3})/2$, we can derive that the number of corresponding polynomials $f$ is at most $O(H\log H)$. 
Hence, combining the above estimates with Lemma \ref{lem:PnH*}, we obtain the desired asymptotic formula for $\# I_2^*(H)$.

(iii)
For any polynomial $f\in I_4^*(H)$ of the form $a_4X^4+a_3X^3+a_2X^2+a_1X+a_0$, assume that $a_4 \ne \pm a_0$ and $f$ is non-degenerate. 
Suppose that the roots of $f$ are $\alpha_1,\alpha_2,\alpha_3,\alpha_4$. 
Applying the same arguments as in the monic case, we can assume that 
$$
\alpha_3\alpha_4 = \zeta \alpha_1\alpha_2 
$$
for some root of unity $\zeta$, which is of at most degree two over $\Q$. 
So, we have 
$$
 \alpha_1^{12} \alpha_2^{12} \alpha_3^{-12} \alpha_4^{-12} = 1.  
$$ 
Then, for $n=4$ we can select in \eqref{jons11} the absolute constant $K=12$ (instead of $K= \lfloor c (\log H)^3\rfloor$). 
Therefore, with $n=4$ and $K=12$ in \eqref{eq:general222}, applying the estimate \eqref{eq:general2} (with the factor $13^{9}$ instead of 
the factor $(c(\log H)^3+1)^9$ containing $\log H$) we obtain the required upper bound $O(H^4)$.

\subsection{Proof of Theorem \ref{thm:reducible2}}

(i) 
First, by Lemma \ref{lem:QnH*} the number of polynomials in $R_2^*(H)$ which are divisible by $X+1$ or $X-1$ 
is equal to $6H^2 + O(H)$. 

For polynomials in $R_2^*(H)$ not divisible by $X+1$ or $X-1$, without loss of generality, 
we only need to count the polynomials $f$ of the form $a_2(X-\alpha_1)(X-\alpha_2)$ with $0 < |\alpha_1\alpha_2| \le 1$ 
(that is, the absolute value of the leading coefficient of $f$ is not less than $|f(0)|$), 
where $0 < |a_2| \le H$ and  $\alpha_1,\alpha_2\in \Q$ are non-zero and not equal to $\pm 1$. 
In fact, if $|\alpha_1\alpha_2| > 1$, then we turn to count their reciprocal polynomials. 

If $|\alpha_1\alpha_2|=1$, then from the proof of Lemma \ref{lem:PnH*} for the case $n=2$, we see that the number of corresponding polynomials is at most 
$O(H\log H)$. 

Now, we assume that $0<|\alpha_1\alpha_2| < 1$. 
Since $\alpha_1$ and $\alpha_2$ are multiplicatively dependent, there exists a positive rational number $0 < b < 1$ such that $\alpha_1= \pm b^k, \alpha_2= \pm b^m$ 
for some non-zero integers $k,m$. Assume that  $k\ge m$. 
Let us write $b=b_1/b_2$, where $b_1,b_2$ are positive integers with $b_1<b_2$ and $\gcd(b_1,b_2)=1$.  
Then, $f$ has the form $a_2X^2 \pm a_2(b^k \pm b^m)X \pm a_2b^{k+m}$, that is, 
$$
f(X)= a_2X^2 \pm a_2(b_1^k/b_2^k \pm b_1^m/b_2^m)X \pm a_2b_1^{k+m}/b_2^{k+m}. 
$$ 
Note that $k+m\ge 1$ due to $0 < |\alpha_1\alpha_2| < 1$. 
Since $a_2b_1^{k+m}/b_2^{k+m}$ is a non-zero integer, we have $b_2^{k+m} \mid a_2$. Since $k+m \ge 1$ and $m \ne 0$, we only have two cases: either $k \ge m >0$, or $k>0>m$. 

If $k \ge m>0$, then $k+m \ge 2$. 
Since $b_2^{k+m} \mid a_2$ and $|a_2| \le H$, we must have
$b_2 \le \sqrt{H}$ and $k+m \le \log_{b_2} H$. The number of $a_2$ divisible by $b_2^{k+m}$
and does not exceeding $H$ is at most $H/b_2^{k+m}$.  Thus,
 the number of corresponding polynomials is at most 
\begin{align*}
 O\Big( \sum_{b_2=2}^{\sqrt{H}} \sum_{b_1=1}^{b_2-1} \sum_{\substack{1 \le m \le k \\ k+m\le \log_{b_2} H }} H/b_2^{k+m} \Big) 
& = \quad  O\Big( H \sum_{b_2=2}^{\sqrt{H}} \sum_{b_1=1}^{b_2-1} \sum_{s=2 }^{\log_{b_2} H} s/b_2^s \Big) \\
& =O\Big(  H \sum_{b_2=2}^{\sqrt{H}}  1/b_2\Big) = O(H\log H). 
\end{align*}

If $k>0>m$, then, since $a_2(b_1^k/b_2^k \pm b_1^m/b_2^m)$ is a non-zero integer, we must have 
$b_2^k \mid a_2$ and $b_1^{-m} \mid a_2$.
In view of $m<0$ and $k+m \ge 1$ we obtain $k\ge 2$. 
Set $m^\prime = -m$. Then, $m^\prime \le k -1$. 
So, the number of corresponding polynomials is at most 
\begin{align*}
 O\Big(  \sum_{b_2=2}^{H} \sum_{b_1=1}^{b_2-1} \sum_{k=2}^{\log_{b_2} H} \sum_{m^\prime = 1}^{k-1} H/(b_2^kb_1^{m^\prime}) \Big) 
&  =  O\Big(  H \sum_{b_2=2}^{H} \sum_{k=2}^{\log_{b_2} H} (k+ \log b_2)/b_2^k \Big) \\
& =O\Big(  H \sum_{b_2=2}^{H}  (\log b_2)/b_2^2\Big) = O(H). 
\end{align*}

Therefore, collecting the above estimates, we obtain 
$$
\# R_2^*(H) = 6H^2 + O( H\log H).
$$

(ii) and (iii). 
First, by Lemma \ref{lem:QnH*}, the number of polynomials in $R_n^*(H)$ which are divisible by $X+1$ or $X-1$ is equal to 
$$
2\nu_{n+1} H^n + O(H^{n-1}). 
$$

Let $f \in R_3^*(H)$ be a polynomial of the form $$(b_1X+b_0)(c_1X+c_0)(d_1X+d_0),$$ where all the coefficients of the involved linear factors are in $\Z$. 
Since $b_0/b_1, c_0/c_1,d_0/d_1$ are multiplicatively dependent, if a prime $p$ divides $b_0b_1c_0c_1d_0d_1$, 
then we must have $p^2 \mid b_0b_1c_0c_1d_0d_1$. 
Let $a$ be the positive integer such that $a^2$ is the maximal square divisor of $|b_0b_1c_0c_1d_0d_1|$. 
Then, $|b_0b_1c_0c_1d_0d_1|$ is a divisor of $a^3$. 
Observing that $|b_1c_1d_1| \le H$ and $|b_0c_0d_0| \le H$, we deduce that $a \le H$. 
Then,  employing \eqref{eq:PnPn} we deduce that the number of such polynomials $f$ is at most 
\begin{align*}
O \Big( \sum_{a=1}^{H} \sum_{b_0,b_1,c_0,c_1,d_0,d_1 \in \cD(a^3)} 1 \Big) 
&= O \Big( \sum_{a=1}^{H} \# \cD(a^3)^6 \Big) \\&= O \Big( \sum_{a=1}^{H} \# \cD(a)^{18} \Big) \\&=
O(H (\log H)^{2^{18}-1}).
\end{align*}

For $n\ge 4$, by Lemma \ref{lem:reducible}, the number of polynomials in $R_n^*(H)$, 
which have a factor of degree at least two and do not have an irreducible factor of degree $n-1$,  
is at most $O(H^{n-2}\log H)$.   

It remains to consider the case when such polynomials have an irreducible factor of degree $n-1$ ($n\ge 3$). 
For any polynomial $f \in R_n^*(H)$ of the form 
$(b_1X+b_0)g$, where $g=c_{n-1}X^{n-1} + \cdots + c_1X +c_0 \in \Z[X]$ is irreducible and 
 $b_1 \ne \pm b_0$, one of the following is true:
either the roots of $g$ are multiplicatively dependent, or $b_0/b_1$ and $c_0/c_{n-1}$ 
(both $b_0/b_1$ and $c_0/c_{n-1}$ are not equal to $0, \pm 1$) are multiplicatively dependent, by Lemma \ref{lem:multQ}. 
In the first case, for $n=3$, by Theorem \ref{thm:irre2} (ii), the number of such polynomials $f$ is at most 
$$
 O\Big( \sum_{b_0=1}^{H}\sum_{b_1=1}^{b_0-1} (H/b_0)^{2} \Big) = O(H^2 \log H).
$$
(Here, we assumed without restriction of generality that $1 \leq b_1<b_0$.)
For $n \ge 4$, by Theorem \ref{thm:irre2} (i), the number of such polynomials $f$ is at most 
\begin{align*}
& O\Big( \sum_{b_0=1}^{H}\sum_{b_1=1}^{b_0-1} (H/b_0)^{n-1} (\log(H/b_0))^{2(n-1)^2-(n-1)-1} \Big) \\
& \quad = O(H^{n-1}(\log H)^{2n^2-5n+2} ), 
\end{align*}
where the factor $(\log H)^{2n^2-5n+2}$ can be removed when $n-1$ is a prime or $n=5$, by Theorem \ref{thm:irre2} (ii) and (iii). 

In the second case, without loss of generality we assume that $|b_0/b_1|>1$ and $|c_0/c_{n-1}|>1$, for
otherwise we can apply all the arguments below to $|b_1/b_0|$ or $|c_{n-1}/c_0|$ (note that we have assumed that both of them are not equal to $1$).  
Then, there exists a positive rational number $r>1$ such that $|b_0/b_1|=r^k, |c_0/c_{n-1}|=r^m$ 
for some positive integers $k,m$. 
We write $r=r_1/r_2$ with positive integers $r_1,r_2$ and $\gcd(r_1,r_2)=1$. 
Then, we have $|b_0|=sr_1^k,|b_1|=sr_2^k, |c_0|=tr_1^m, |c_{n-1}|=tr_2^m$ for some positive integers $s,t$. 
Since  $|b_0c_0|=|f(0)| \le H$, we have $str_1^{k+m} \le H$. 
Let $a=|f(0)|$. Then, $r_1,s,t$ are divisors of $a$. If we fix $a,r_1,s,k,m$, then $t$ is also fixed up to a sign. 
Hence, the number of such polynomials $f$ is at most 
$$ O\Big( \sum_{a = 1}^{H}\sum_{r_1,s\in \cD(a)} \sum_{r_2=1}^{r_1-1} \sum_{k, m=1}^{\log_{r_1} H} (H/(sr_1^k))^{n-2} \Big).$$
Since $\sum_{k, m=1}^{\log_{r_1} H} (H/(sr_1^k))^{n-2}=
O(H^{n-2} (sr_1)^{2-n} \log H)$, we can further bound this by
\begin{equation}  \label{eq:final}
  O\Big( H^{n-2} \log H \sum_{a = 1}^{H}\sum_{r_1\in \cD(a)} \frac{1}{r_1^{n-3}} \sum_{s\in \cD(a)} \frac{1}{s^{n-2}} \Big). 
\end{equation}

Assume first that $n=3$. Then, the estimate \eqref{eq:final} becomes 
\begin{align*}
& O\Big( H \log H \sum_{a = 1}^{H} \# \cD(a) \sum_{s\in \cD(a)} \frac{1}{s} \Big) 
 =  O\Big( H \log H\sum_{a = 1}^{H} \# \cD(a) \log a \Big)  \\
& \quad  =  O\Big( H (\log H)^2 \sum_{a = 1}^{H} \# \cD(a) \Big)    =  O\big( H^2 (\log H)^3 \big), 
\end{align*}
where we use 
\eqref{eq:PnPn}. 

For $n=4$ the estimate \eqref{eq:final} becomes 
\begin{align*}
& O\Big( H^2 \log H \sum_{a = 1}^{H} \sum_{r_1\in \cD(a)} \frac{1}{r_1} \sum_{s\in \cD(a)} \frac{1}{s^2} \Big) 
 =  O\Big( H^2 \log H\sum_{a = 1}^{H} \sum_{r_1\in \cD(a)} \frac{1}{r_1} \Big)  \\
& \quad  =  O\Big( H^2 \log H \sum_{a = 1}^{H} \sigma(a)/a \Big)    =  O\big(H^3 \log H\big), 
\end{align*}
where $\sigma(a)=\sum_{d\mid a} d$ and, as is well-known, $\sum_{a = 1}^{H} \sigma(a)/a=O(H)$. 

Finally, for $n\ge 5$, it is easy to see that the estimate \eqref{eq:final} yields $O(H^{n-1}\log H)$.  

Combining all the above estimates, we obtain 
$$
\# R_3^*(H) = 2\nu_4 H^3 + O(H^2(\log H)^3), 
$$
and for $n\ge 4$, 
$$
\# R_n^*(H) = 2\nu_{n+1} H^n + O(H^{n-1}(\log H)^{2n^2-5n+2} ), 
$$
where the factor $(\log H)^{2n^2-5n+2}$ can be replaced by $\log H$ when $n-1$ is a prime or $n=5$.

\section*{Acknowledgements} 

The authors would like to thank Igor E. Shparlinski for introducing them into this topic. The research of Art\=uras Dubickas was funded by a grant (No. S-MIP-17-66/LSS-110000-1274) 
from the Research Council of Lithuania.
The research of Min Sha was supported by the Macquarie University Research Fellowship.

\end{document}